\documentclass[12pt,leqno]{amsart}
\usepackage{times}
\usepackage{amssymb}
\usepackage[dvips]{graphicx}

\allowdisplaybreaks
\usepackage{amsmath,amssymb,txfonts,color}
\usepackage{amssymb}
\usepackage{amsxtra}
\usepackage{amsmath}
\usepackage{txfonts}
  
\def\ball{B^n_2}
\def\bbR{\mathbb R}
\newtheorem{thm}{Theorem}

\newtheorem{cor}[thm]{Corollary}
\newtheorem{prop}[thm]{Proposition}
\newtheorem{rem}{Remark}

\begin{document}

\title{Anisotropic Sobolev Capacity with Fractional Order}
 
\author{Jie Xiao} \address{Department of Mathematics and Statistics, Memorial University, St. John's, NL A1C 5S7, Canada}
\email{jxiao@mun.ca}\thanks{Research of JX is supported in part by: NSERC and URP of MUN, Canada.}

 \author{Deping Ye}
\address{Department of Mathematics and Statistics, Memorial University, St. John's, NL A1C 5S7, Canada}
 \email{deping.ye@mun.ca}
 \thanks{Research of DY is supported by NSERC, Canada}

  

\begin{abstract} In this paper, we introduce the  anisotropic Sobolev capacity with fractional order and develop some basic properties for this new object. Applications to the theory of anisotropic fractional Sobolev spaces are provided. In particular, we give geometric  characterizations for a nonnegative Radon measure $\mu$ that naturally induces an embedding of the anisotropic fractional Sobolev class $\dot{\Lambda}_{\alpha,K}^{1,1}$ into the $\mu$-based-Lebesgue-space $L^{n/\beta}_\mu$ with $0<\beta\le n$. Also, we investigate the anisotropic fractional $\alpha$-perimeter. Such a geometric quantity can be used to approximate the  anisotropic Sobolev capacity with fractional order. Estimation on the constant in the related Minkowski inequality, which is asymptotically optimal as $\alpha\rightarrow 0^+$, will be provided.  
\end{abstract}

\maketitle

\section{Anisotropic fractional Sobolev capacity}\label{s1}

A subset $K\subset \bbR^n$ is said to be a convex body if $K$ is a convex compact subset of $\bbR^n$ with nonempty interior. Related to each convex body $K$ with the origin in its interior, one can uniquely define the support function $h_K(\cdot): S^{n-1}\rightarrow \bbR$ as $$h_K(u)=\max\{\langle y, u\rangle, \ \ \ y\in K\}, \ \ \ \forall u\in S^{n-1}, $$ where  $\langle\cdot, \cdot\rangle$ denotes the usual inner product on $\bbR^n$ and induces the usual Euclidean norm  $\|\cdot \|$.  The unit Euclidean ball of $\bbR^n$ is $\ball =\{x\in \bbR^n: \|x\|\leq 1\}$.   For a subset $L\subset \bbR^n$ with the origin in $L$, its polar $L^*$ is defined by $L^*=\{y\in \bbR^n: \ \langle x, y\rangle \leq 1, \ \forall x\in L\}.$ Note that $L^*$ is always convex no matter the convexity of $L$.  The volume of $K$ is denoted by $V(K)$, and more general, $V(M)$ denotes the appropriate dimensional Hausdorff content of $M$.  For a subset $E\subset \bbR^n$, $\overline{E}$ denotes the closure of $E$.   

The Minkowski functional of $K$ is denoted by $\|\cdot\|_K$ and is defined as
\begin{eqnarray*} \|x\|_K=\inf\{\lambda>0: x\in \lambda K\},
\end{eqnarray*} where  $\lambda K=\{\lambda y:  y\in K\}$ 
for $\lambda\in \bbR$. In particular, if $K=-K$, then $K$ is said to be origin-symmetric.  It is easy to check that, for any origin-symmetric convex body $K\subset \bbR^n$,  $\|\cdot\|_K$ defines a norm on $\bbR^n$.  The usual Euclidean norm $\|\cdot\|$ is related to $K=\ball$. 

Throughout this paper,   $\alpha\in (0, 1)$ is a constant and $K\subset\mathbb R^n$ is always assumed to be an origin-symmetric convex body. A function $f$ is said to be of $C_0^\infty$, denoted  by $f\in C_0^\infty$, if $f$ is smooth and has compact support in $\mathbb R^n$.  Consider the following norm for $f\in C^\infty_0$    \begin{equation*}
\|f\|_{\dot{\Lambda}_{\alpha, K} ^{1,1}}=\int_{\mathbb R^n} \int_{\mathbb R^n}\frac{|f(x)-f(y)|}{\|x-y\|_K^{n+\alpha}}\,dx\,dy.
\end{equation*}  The completion of all functions $f\in
C^\infty_0$ with the above norm is denoted by $\dot{\Lambda}^{1,1}_{\alpha, K}$.  Such a function space will be called the anisotropic fractional Sobolev  space with respect to $K$ (or the homogeneous $(\alpha,1,1, K)$-Besov space).   Theorems 1 and 2 in \cite{L1} imply that 
\begin{equation}\label{function-inequality-1} 
\lim_{\alpha\to 0^+}\alpha \|f\|_{\dot{\Lambda}^{1,1}_{\alpha, K}}=2nV(K)  \|f\|_{L^1}\quad\&\quad \lim_{\alpha\to 1^-}(1-\alpha)\|f\|_{\dot{\Lambda}^{1,1}_{\alpha, K}}=\int_{\bbR^n} \|\nabla f(x) \|_{Z_1^*K}\,dx,
\end{equation}  
where $Z_1^*K$ is the polar body of  $Z_1K$ (the moment body of $K$) and the support function of  $Z_1K$ is determined by   $$h_{Z_1K}(x)=\|x\|_{Z_1^*K}=\frac{n+1}{2}\int_K |\langle x, y\rangle|\,dy, \ \ \forall x\in \bbR^n.$$    The case $K$ being the unit Euclidean ball $\ball$ has been considered in, e.g., \cite{ BBM1, BBM2, L1, MS1, MS2}. 
 
For any given compact subset $L$ of $\bbR^n$, one can define  ${\mathrm{cap}}(L; \dot{\Lambda}_{\alpha, K}^{1,1})$, the anisotropic fractional Sobolev capacity of $L$ with respect to $K$,   by 
\begin{equation}
{\mathrm{cap}}(L;
\dot{\Lambda}_{\alpha, K}^{1,1})=\inf\big\{\|f\|_{\dot{\Lambda}_{\alpha, K}^{1,1}}:\ \ 
f\in C^\infty_0\ \ \&\ \ f\ge \mathbf{1}_L\big\}. \label{Related capacity-0}
\end{equation} 
Hereafter, $\mathbf{1}_E$ denotes the indicator function of $E\subset\mathbb R^n$. For any compact $L\subset\mathbb R^n$, formula (\ref{function-inequality-1}) implies,  (see also \cite{L2}),
\begin{equation}\label{Sobolev:limit--1}
\lim_{\alpha\to 0^+}\alpha \ \mathrm{cap}(L; \dot{\Lambda}_{\alpha, K} ^{1,1})
=2nV(L) V(K)\quad\&\quad \lim_{\alpha\to 1^-}(1-\alpha)\ \mathrm{cap}(L; \dot{\Lambda}_{\alpha, K}^{1,1})=  \mathrm{cap}(L; \dot{W}_K^{1,1}),
\end{equation}  where
\begin{equation*}   {\mathrm{cap}}(L;
\dot{W}_K^{1,1})=\inf \left\{\int_{\bbR^n} \|\nabla f(x) \|_{Z_1^*K}\,dx
:\quad
f\in C^\infty_0\quad\&\quad f\ge \mathbf{1}_L \right\}.
\end{equation*}

For general subset $E\subset \bbR^n$,  the anisotropic fractional Sobolev capacity (or the homogeneous end-point Besov capacity) of $E$ with respect to $K$,  denoted by  $\mathrm{cap}(E; \dot{\Lambda}_{\alpha, K}^{1,1})$, can be defined by  
\begin{equation} 
\mathrm{cap}(E; \dot{\Lambda}_{\alpha, K}^{1,1})=\inf_{\mathrm{open}\ O\supseteq E}\mathrm{cap}(O;\dot{\Lambda}_{\alpha, K}^{1,1})=\inf_{\mathrm{open}\ O\supseteq E}\left(\sup_{\mathrm{compact}\ L\subseteq O}\mathrm{cap}(L;\dot{\Lambda}_{\alpha, K}^{1,1})\right). \label{Related capacity-1} 
\end{equation}   Similarly, for general subset $E\subset \bbR^n$, 
$$
\mathrm{cap}(E; \dot{W}_K^{1,1})=\inf_{\mathrm{open}\ O\supseteq E}\mathrm{cap}(O;\dot{W}_K^{1,1})=\inf_{\mathrm{open}\ O\supseteq E}\left(\sup_{\mathrm{compact}\ L\subseteq O}\mathrm{cap}(L;\dot{W}_K^{1,1})\right).
$$
{See also \cite{Adams, Adams09, AdamsXiao, S, Xiao06, Xiao14a} for special case $K=\ball$}.

As a natural outcome of exploring some essential links between  \cite{Xiao06,Xiao14a} and \cite{L1,L2}, this paper will focus on the above-newly-introduced anisotropic fractional Sobolev capacity,  in particular, its immediate applications to the embedding/trace theory of the anisotropic Sobolev space with fractional order.  Section \ref{s2} is dedicated to some intrinsic properties of the anisotropic Sobolev capacity with fractional order.  Section \ref{extrinsic s3} is for the extrinsic nature of the anisotropic Sobolev capacity with fractional order via the so-called anisotropic fractioal perimeter. Moreover, estimation on the constant in the related Minkowski inequality, which is asymptotically optimal as $\alpha\rightarrow 0^+$,  will be provided.  The anisotropic fractional Sobolev inequalities and  their geometric counterparts for anisotropic fractional capacity will be discussed in Section \ref{s3}.

 \section{Intrinsic properties}\label{s2}
 
 We begin with exploring some intrinsic properties of the anisotropic Sobolev capacity with fractional order.

\begin{thm}\label{t1} The set-function $E\mapsto \mathrm{cap}(E; \dot{\Lambda}_{\alpha, K}^{1,1})$ is nonnegative and has the following properties.  \\
\noindent {\rm (i)}  Homogeneity:  let $r>0$ be a real constant, then 
 \begin{eqnarray*} \mathrm{cap}(rE; \dot{\Lambda}_{\alpha, K}^{1,1})=r^{n-\alpha}\mathrm{cap}(E; \dot{\Lambda}_{\alpha, K}^{1,1}), \ \ \mathrm{and} \ \    \mathrm{cap}(E; \dot{\Lambda}_{\alpha, rK}^{1,1})= r^{n+\alpha}\mathrm{cap}(E; \dot{\Lambda}_{\alpha, K}^{1,1}) .\end{eqnarray*} 
 Moreover, for all $r, s>0$, 
 $$
 \mathrm{cap}(sE; \dot{\Lambda}_{\alpha, rK}^{1,1}) = s^{n-\alpha} r^{n+\alpha}\mathrm{cap}(E; \dot{\Lambda}_{\alpha, K}^{1,1}).
 $$

\vskip 2mm \noindent{\rm (ii)} Monotonicity: for all subsets $E_1\subseteq E_2\subseteq\mathbb R^n$, one has $$\mathrm{cap}(E_1; \dot{\Lambda}_{\alpha, K}^{1,1})\le\mathrm{cap}(E_2; \dot{\Lambda}_{\alpha, K}^{1,1}).$$

\noindent {\rm (iii)} Subaddivity: for all compact sets $L_1, L_2\subseteq\mathbb R^n$, one has $$\mathrm{cap}(L_1\cup L_2;\dot{\Lambda}_{\alpha, K}^{1,1})
\le\mathrm{cap}(L_1; \dot{\Lambda}_{\alpha, K}^{1,1})+\mathrm{cap}(L_2; \dot{\Lambda}_{\alpha, K}^{1,1}). $$

\noindent {\rm (vi)} Upper-semi-continuity: for all decreasing sequence  $\{L_j\}_{j=1}^\infty$  of compact subsets of $\bbR^n$ with $L_1\supseteq L_2\supseteq L_3\supseteq\cdots$, one has
  $$\lim_{j\to\infty}\mathrm{cap}(L_j; \dot{\Lambda}_{\alpha, K}^{1,1})=\mathrm{cap}(\cap_{j=1}^\infty L_j; \dot{\Lambda}_{\alpha, K}^{1,1}).$$ 
\end{thm}

\begin{proof} (i) Let $r> 0$.  First,  the desired equality 
$
\mathrm{cap}(E; \dot{\Lambda}_{\alpha, rK}^{1,1})=r^{n+\alpha}\mathrm{cap}(E; \dot{\Lambda}_{\alpha, K}^{1,1}) 
$ 
 follows  immediately from  
 $
 \|x-y\|_{rK}=r^{-1}\|x-y\|_K$ for all $x, y\in \bbR^n.$  
  
 To prove 
  $
 \mathrm{cap}(rE; \dot{\Lambda}_{\alpha, K}^{1,1})=r^{n-\alpha}\mathrm{cap}(E; \dot{\Lambda}_{\alpha, K}^{1,1}),
 $ 
 it is enough to prove the equality for compact sets, due to equation (\ref{Related capacity-1}). Consider $\|g\|_{\dot{\Lambda}_{\alpha, K} ^{1,1}}$ with $g(x)=f(rx)$ as follows:  \begin{eqnarray*}
\|g\|_{\dot{\Lambda}_{\alpha, K} ^{1,1}}&=&\int_{\mathbb R^n} \int_{\mathbb R^n}\frac{|g(x)-g(y)|}{\|x-y\|_K^{n+\alpha}}\,dx\,dy\\ &=&\int_{\mathbb R^n} \int_{\mathbb R^n}\frac{|f(rx)-f(ry)|}{\|rx-ry\|_K^{n+\alpha}}r^{\alpha-n} \,d(rx)\,d(ry)\\ &=&\int_{\mathbb R^n} \int_{\mathbb R^n}\frac{|f(x)-f(y)|}{\|x-y\|_K^{n+\alpha}}r^{\alpha-n} \,dx\,dy\\
&=&r^{\alpha-n} \|f\|_{\dot{\Lambda}_{\alpha, K} ^{1,1}}.
\end{eqnarray*} 
Hence, for all compact set $L\subset \bbR^n $, one has  
\begin{eqnarray*} \mathrm{cap}(rL, \dot{\Lambda}_{\alpha, K}^{1,1}) &=& \inf\big\{\|f\|_{\dot{\Lambda}_{\alpha, K}^{1,1}}:\quad
f\in C^\infty_0\quad\&\quad f\ge \mathbf{1}_{rL} \big\} \\ &=&  \inf\big\{r^{n-\alpha} \ \|g\|_{\dot{\Lambda}_{\alpha, K}^{1,1}}:\quad
f\in C^\infty_0\quad\&\quad g\ge \mathbf{1}_{L} \big\}\\
&=&r^{n-\alpha} \mathrm{cap}(L, \dot{\Lambda}_{\alpha, K}^{1,1}). 
\end{eqnarray*}  
 
 Finally, for all $r, s>0$, one has $$\mathrm{cap}(sE; \dot{\Lambda}_{\alpha, rK}^{1,1}) =s^{n-\alpha} \mathrm{cap}(E; \dot{\Lambda}_{\alpha, rK}^{1,1}) = s^{n-\alpha} r^{n+\alpha}\mathrm{cap}(E; \dot{\Lambda}_{\alpha, K}^{1,1}).$$

\vskip 2mm \noindent   (ii) It is enough to prove the monotonicity for compact sets, again due to equation (\ref{Related capacity-1}). For two compact sets $L_1$ and $L_2$ with $L_1\subset L_2$, it is easily checked that $$\{f\in C^{\infty}_0: f\geq \mathbf{1}_{L_1}\}\supset \{f\in C^{\infty}_0: f\geq \mathbf{1}_{L_2}\}.$$ Hence,  \begin{eqnarray*} \mathrm{cap}(L_1, \dot{\Lambda}_{\alpha, K}^{1,1}) &=& \inf\big\{\|f\|_{\dot{\Lambda}_{\alpha, K}^{1,1}}:\ 
f\in C^\infty_0\ \&\  f\ge \mathbf{1}_{L_1} \big\} \\ &\leq& \inf\big\{\|f\|_{\dot{\Lambda}_{\alpha, K}^{1,1}}:\ 
f\in C^\infty_0\ \&\  f\ge \mathbf{1}_{L_2} \big\}\\
&=&\mathrm{cap}(L_2, \dot{\Lambda}_{\alpha, K}^{1,1}) . 
\end{eqnarray*}

\vskip 2mm \noindent (iii) Without loss of generality, we may assume $\mathrm{cap}(L_j; \dot{\Lambda}_{\alpha, K}^{1,1})<\infty$ with $j=1,2$, as otherwise the consequence holds true trivially.  For any $\epsilon>0$, there are $f_1,f_2\in C^\infty_0$ such that 
$$
f_j\ge \mathbf{1}_{L_j}\ \ \&\ \ \|f_j\|_{\dot{\Lambda}_{\alpha, K}^{1,1}}<\mathrm{cap}(L_j; \dot{\Lambda}_{\alpha, K}^{1,1})+\epsilon, \quad\forall  j=1,2.
$$
Let $f=\max\{f_1,f_2\}\in C^\infty_0$ and clearly the function $f$ satisfies 
$$
f\ge \mathbf{1}_{L_1\cup L_2}\ \ \&\ \ |f(x)-f(y)|\le|f_1(x)-f_1(y)|+|f_2(x)-f_2(y)|,\ \ \forall  x,y\in\mathbb R^n.
$$
This further implies
$$
\mathrm{cap}(L_1\cup L_2; \dot{\Lambda}_{\alpha, K} ^{1,1})\le \|f\|_{\dot{\Lambda}_{\alpha, K}^{1,1}}\le  \|f_1\|_{\dot{\Lambda}_{\alpha, K}^{1,1}}+ \|f_2\|_{\dot{\Lambda}_{\alpha, K}^{1,1}}\le \mathrm{cap}(L_1; \dot{\Lambda}_{\alpha, K}^{1,1})+\mathrm{cap}(L_2; \dot{\Lambda}_{\alpha, K}^{1,1})+2\epsilon.
$$ 
The desired consequence follows by letting $\epsilon\rightarrow 0$. 

\vskip 2mm \noindent (iv) Suppose that $\{L_j\}_{j=1}^\infty$ is a decreasing sequence of compact subsets of $\mathbb R^n$. Then,
$L=\cap_{j=1}^\infty L_j$ is compact. For any $\epsilon\in (0,1)$, there is a function $f\in C_0^\infty$ such that
$$
f\ge \mathbf{1}_{L}\quad\&\quad \|f\|_{\dot{\Lambda}_{\alpha, K}^{1,1}}
< \mathrm{cap}(L; \dot{\Lambda}_{\alpha, K} ^{1,1})+\epsilon.
$$ 
Let 
$
L_{f, \epsilon}=: \{x\in\mathbb R^n: f(x)\ge 1-\epsilon\},
$
which is compact.  Due to $L_j$ decreasing to $L$, one can find an integer $j>0$ large enough, such that,  $L_j\subset L_{f, \epsilon}.$  By Part (ii) and formula (\ref{Related capacity-0}), one has, 
$$
\lim_{j\to\infty}\mathrm{cap}(L_j; \dot{\Lambda}_{\alpha, K}^{1,1})\le \mathrm{cap}\big(L_{f, \epsilon}; \dot{\Lambda}_{\alpha, K}^{1,1}\big)\le{(1-\epsilon)^{-1}\|f\|_{\dot{\Lambda}_{\alpha, K} ^{1,1}}}\le \frac{\mathrm{cap}(L; \dot{\Lambda}_{\alpha, K} ^{1,1})+\epsilon}{1-\epsilon}.
$$
Letting $\epsilon\to 0$ and again by Part (ii), we get
$$
\mathrm{cap}(L; \dot{\Lambda}_{\alpha, K}^{1,1})\le\lim_{j\to\infty}\mathrm{cap}(L_j; \dot{\Lambda}_{\alpha, K}^{1,1})\le \mathrm{cap}(L; \dot{\Lambda}_{\alpha, K}^{1,1}),
$$
and hence equality holds.
\end{proof}

 \begin{rem}
\label{r21} Along similar lines, one can  prove  analogous intrinsic result for the anisotropic Sobolev capacity $\mathrm{cap}(\cdot; \dot{W}^{1,1}_K)$, with  $\dot{\Lambda}^{1,1}_{\alpha,K}$ and $n\pm\alpha$ in Theorem \ref{t1} replaced by $\dot{W}^{1,1}_K$ and $n\pm 1$ respectively. 
\end{rem} 

\section{Extrinsic properties}\label{extrinsic s3}
 In this section, we will reveal  an extrinsic nature of the anisotropic Sobolev capacity with fractional order via the so-called anisotropic fractional perimeter.
 
For a set $E\subseteq\mathbb R^n$, let $E^c=\mathbb R^n\setminus E$ be the complement of $E\subset \bbR^n$. Define $P_\alpha(E,K)$, the anisotropic fractional $\alpha$-perimeter of $E$ with respect to $K$ \cite {L2},  as 
$$ P_\alpha(E, K)
=\int_{E}\int_{E^c}\frac{1}{\|x-y\|_K^{n+\alpha}}\,dxdy=\frac{\|\mathbf{1}_E\|_{\dot{\Lambda}_{\alpha, K}^{1,1}}}{2}.$$  
 Theorems 4 and 6 in \cite{L2} assert that, if $E\subset \bbR^n$ is a bounded Borel set of finite perimeter, then 
\begin{equation}\label{Monika-L2-1}
\lim_{\alpha\to 0^+}\alpha P_\alpha(E, K)=n V(E) V(K) \quad\&\quad \lim_{\alpha\to 1^-}(1-\alpha)P_\alpha(E, K)=P(E, Z_1K). 
\end{equation} 
 Here and henceforth,  $P(E,F)$ stands for
 the anisotropic perimeter of a Borel set $E\subset \bbR^n$ with respect to an origin-symmetric convex body $F$, which has the following form:   
 $$
 P(E, F)=\int _{\partial ^* E} \|\nu _E(x)\|_{F^*}\,d\mathcal{H}^{n-1}(x),
 $$ 
 with $\mathcal{H}^{n-1}$ the $(n-1)$ dimensional Hausdorff measure, $\nu_E(x)$ the measure theoretic outer unit normal of $E$ at point $x$ in  $\partial ^*E$, the reduced boundary of $E$. In particular, $P(E)=P(E, \ball)$ is called the perimeter of $E$. When $\partial E$, the boundary of $E$, is smooth, $P(E)$ is equal to the usual surface area of $\partial E$. On the other hand, $P(E, F)$ equals the classical mixed volume of $E$ and $F$,  if $E$ is also a convex body.  The special case $P_\alpha(E)=P_\alpha(E, \ball)$, named as the fractional $\alpha$-perimeter of $E$ (cf.  \cite{FMM}), is a classical object and receives  a lot of attention. In particular, by formula (\ref{Monika-L2-1}), one has, 
$$
\lim_{\alpha\to 0^+}\alpha P_\alpha(E)=n V(\ball) V(E)\quad\&\quad \lim_{\alpha\to 1^-}(1-\alpha)P_\alpha(E)=2^{-1}\tau_n P(E),
$$
where $\tau_n=\int _{\mathbb S^{n-1}}|\cos(\theta)|\,d\sigma$ with $\theta$ being the angle deviation from the vertical direction and $d\sigma$ being the standard area measure on  the unit sphere $\mathbb S^{n-1}$ of $\mathbb R^n$; see \cite{MS1, MS2}. 

The following cyclic inequality for the anisotropic fractional perimeters holds.  
\begin{prop} Let $0<\alpha<\beta<\gamma<1$. For all $E\subset \bbR^n$,  one has, $$ \big[P_{\beta}(E, K)\big]^{\gamma-\alpha}\leq \big[P_{\alpha}(E, K)\big]^{  \gamma-\beta} \  \big[P_{\gamma}(E, K)\big]^{\beta-\alpha }. $$  \end{prop}
\begin{proof} Let $0<\alpha<\beta<\gamma<1$ which implies $0<\frac{\beta-\alpha}{\gamma-\alpha}<1$.  By H\"{o}lder's inequality, one has,  \begin{eqnarray*} 
 P_\beta(E, K)
  &=& \int_{E}\int_{E^c}\frac{1}{\|x-y\|_K^{n+\beta}}\,dxdy \nonumber \\ &=&  \int_{E}\int_{E^c} \left(\frac{1}{\|x-y\|_K^{n+\alpha}}\right)^{\frac{\gamma-\beta}{\gamma-\alpha}} \left(\frac{1}{\|x-y\|_K^{n+\gamma}}\right)^{\frac{\beta-\alpha}{\gamma-\alpha}}\,dx\,dy \nonumber \\&\leq&  \left(  \int_{E}\int_{E^c} \frac{1}{\|x-y\|_K^{n+\alpha}} \,dx\,dy\right)^{\frac{\gamma-\beta}{\gamma-\alpha}}    \left( \int_{E}\int_{E^c}  \frac{1}{\|x-y\|_K^{n+\gamma}} \,dx\,dy \right)^{\frac{\beta-\alpha}{\gamma-\alpha}} \nonumber \\&=& \big(P_\alpha(E, K)\big)^{\frac{\gamma-\beta}{\gamma-\alpha}}  \ \big( P_\gamma(E, K)\big)^{\frac{\beta-\alpha}{\gamma-\alpha}}. \end{eqnarray*} The desired inequality follows by taking power $\gamma-\alpha$ from both sides. \end{proof}

For bounded open set $E\subset \bbR^n$  with  $V(\partial E)=V(\overline{E}\setminus E)=0$, one has  \begin{eqnarray}\label{smooth-perimater}P_{\alpha}(\overline{E}, K)= P_{\alpha}( {E}, K).\end{eqnarray}  In fact, for all (fixed)  $y\in E\cup \overline{E}^c$, there is $r>0$, such that $\|y-x\|_K>r$ for all $x\in \overline{E}\setminus E$ as $E\cup\overline{E}^c$ is open. Hence, for all $y\in E\cup \overline{E}^c$,  \begin{eqnarray*} 0\leq  \int_{\overline{E}\setminus E} \frac{1}{\|x-y\|_K^{n+\alpha}}\,dx  \leq  \int_{\overline{E}\setminus E} \frac{1}{r^{n+\alpha}}\,dx=\frac{V(\overline{E}\setminus E)}{r^{n+\alpha}}=0. \end{eqnarray*}   This further implies that 
$$\int_{\overline{E}^c} \left(\int_{\overline{E}\setminus E}\frac{1}{\|x-y\|_K^{n+\alpha}}\,dx\right)dy=\int_{E}\left(\int_{\overline{E}\setminus E} \frac{1}{\|x-y\|_K^{n+\alpha}}\,dx\right) dy=0,$$  and thus, the desired formula (\ref{smooth-perimater}) holds: \begin{eqnarray*} 
P_{\alpha}(\overline{E}, K)- P_{\alpha}( {E}, K) &=& \int_{\overline{E}}\int_{\overline{E}^c}\frac{1}{\|x-y\|_K^{n+\alpha}}\,dxdy- \int_{E}\int_{E^c}\frac{1}{\|x-y\|_K^{n+\alpha}}\,dxdy\\ &=&  \int_{\overline{E}^c} \left(\int_{\overline{E}\setminus E}\frac{1}{\|x-y\|_K^{n+\alpha}}\,dx\right)dy- \int_{E}\left(\int_{\overline{E}\setminus E} \frac{1}{\|x-y\|_K^{n+\alpha}}\,dx\right) dy\\
&=&0.
\end{eqnarray*}  

 Similar to the proof of Theorem \ref{t1}, $P_{\alpha}(E, K)$ has the following homogeneity:  for all $r, s>0$, \begin{eqnarray} \label{homogeneity P(E,K)}  P_{\alpha}(sE, rK)=s^{n-\alpha} r^{n+\alpha} P_{\alpha}(E, K). \end{eqnarray}  It is known that $P_{\alpha}(E, K)\geq \gamma_{\alpha} (K) V(E)^{\frac{n-\alpha}{n}}$ holds true for every bound Borel set $E\subset \bbR^n$  with $\gamma_{\alpha}(K)>0$ a constant defined by  (cf. \cite{L2})  
\begin{equation} 
\label{fractional minkowski001}\gamma_{\alpha} (K)=\inf\{P_{\alpha}(E, K) V(E)^{-\frac{n-\alpha}{n}}: \ E\subset \Omega, V(E)>0\},
\end{equation}  
where $\Omega $ is a given and fixed open bounded subset of $\bbR^n$.   As claimed in   \cite{L2}, the constant $\gamma_{\alpha}(K)$ defined in formula (\ref{fractional minkowski001}) only depends on $K$ and is independent of  the choice of $\Omega$.  Heuristically,  formula (\ref{homogeneity P(E,K)}) indicates that ${\gamma_{\alpha}(K)}{V(K)^{-\frac{n+\alpha}{n}}}$ may be even independent of $K$. 
 
 Following the idea of verifying \cite[Lemma 6.1]{DPV}, we establish the following anisotropic isoperimetric inequality for $P_{\alpha}(E, K)$, which provides an estimate  for the constant $\gamma_{\alpha}(K)$.   
\begin{thm} \label{Minkowski type-1} Let $E$ be a bounded Borel subset of $\bbR^n$. The following  anisotropic isoperimetric inequality with fractional order $\alpha\in (0,1)$ holds: 
$$
\alpha P_{\alpha}(E, K)\geq {n} V (K)^{\frac{n+\alpha}{n}}  V(E)^{\frac{n-\alpha}{n}}.
$$ 
  Moreover, this inequality is asymptotically optimal in the sense of 
 $$
\lim_{\alpha\rightarrow 0^+} \alpha P_{\alpha}(E, K)=\lim_{\alpha\rightarrow 0^+} {n} V (K)^{\frac{n+\alpha}{n}}  V(E)^{\frac{n-\alpha}{n}}={n} V (K)  V(E). 
$$ 
\end{thm}  

\begin{proof} Let $E$ be a bounded Borel subset of $\bbR^n$. The desired inequality holds trivially if $V(E)=0$. Now let us consider $0<V(E)<\infty$, and let $r=\left(\frac{V(E)}{V(K)}\right)^{1/n}>0$. For any fixed $x\in E$, let 
$$
B_r(x)=\{z\in \bbR^n: \|z-x\|_K\leq r\}.
$$ 
In fact, the volume of $K$ is equal to $V(\{z: \|z\|_K\leq 1\})$ and hence the volume of $B_r(x)$ equals $V(E)$. This further implies 
\begin{eqnarray*}
V(E^c\cap B_r(x))&=& V(B_r(x)\setminus E)\\
&=&V(B_r(x))-V(E\cap B_r(x))\\ &=&V(E)-V(E\cap B_r(x))\\
&=&V(E\setminus B_r(x))\\
&=& V(B_r(x)^c\cap E). 
\end{eqnarray*}
Note that 
$\|y-x\|_K \leq r$ for $y\in E^c\cap B_r(x)$ and $\|y-x\|_K>r$ for $y\in B_r(x)^c\cap E$. 
Thus,  
\begin{eqnarray*}    
\int_{E^c\cap B_r(x)}\frac{dy}{\|x-y\|_K^{n+\alpha}}   &\geq&  \int_{E^c\cap B_r(x)}\frac{dy}{r^{n+\alpha}}\\
&=& \frac{V(E^c\cap B_r(x))}{r^{n+\alpha}}\\
& =&  \frac{V(B_r(x)^c\cap E)}{r^{n+\alpha}}\\ 
&=&  \int_{B_r(x)^c\cap E}\frac{dy}{r^{n+\alpha}}\\
&\geq& \int_{B_r(x)^c\cap E}\frac{dy}{\|x-y\|_K^{n+\alpha} }.  
\end{eqnarray*} 
This in turn implies   
\begin{eqnarray*} 
\int_{E^c}\frac{dy}{\|x-y\|_K^{n+\alpha}} &=&  \int_{E^c\cap B_r(x)}\frac{dy}{\|x-y\|_K^{n+\alpha}}+ \int_{E^c\cap B_r(x)^c}\frac{dy}{\|x-y\|_K^{n+\alpha}}  \\ &\geq&  \int_{B_r(x)^c\cap E}\frac{dy}{\|x-y\|_K^{n+\alpha} }+ \int_{E^c\cap B_r(x)^c}\frac{dy}{\|x-y\|_K^{n+\alpha}}\\
&=& \int_{ B_r(x)^c}\frac{dy}{\|x-y\|_K^{n+\alpha}},
\end{eqnarray*}  
where the last integral can be calculated by Fubini's theorem as follows: 
\begin{eqnarray*}
\int_{ B_r(x)^c}\frac{dy}{\|x-y\|_K^{n+\alpha}} &=& \int_{ \{y: \|y-x\|_K>r\}}\frac{dy}{\|x-y\|_K^{n+\alpha}}\\ &=& \int_{ \{y: \|y-x\|_K>r\}}\left( \int_{\|y-x\|_K}^{\infty} (n+\alpha)t^{-n-\alpha-1}\,dt\right)\,dy\\  &=&  \int_r^{\infty}(n+\alpha)t^{-n-\alpha-1} \left(\int_{ \{y: r< \|y-x\|_K\leq t\}} \,dy\right)\,dt\\ &=& V(K) \int_r^{\infty}(n+\alpha)t^{-n-\alpha-1} \left(t^n-r^n\right)\,dt\\ 
&=& \frac{n }{\alpha} \cdot r^{-\alpha} V(K)\\
&=&\frac{n }{\alpha} \cdot  \frac{V(K)^{1+\alpha/n}}{V(E)^{\alpha/n}}. 
\end{eqnarray*} 
Hence, one gets  
\begin{eqnarray*}P_{\alpha}(E, K)= \int _E \left(\int_{E^c}\frac{dy}{\|x-y\|_K^{n+\alpha}}\right)\,dx \geq \int_E \left(\int_{ B_r(x)^c}\frac{dy}{\|x-y\|_K^{n+\alpha}} \right)\,dx\geq   \frac{n }{\alpha} \cdot {V(K)^{\frac{n+\alpha}{n}}}{V(E)^{\frac{n-\alpha}{n}}}.  
\end{eqnarray*} 
The asymptotic optimality is a direct consequence of formula (\ref{Monika-L2-1}), i.e.,  $$n V(E)V(K)=\lim_{\alpha\rightarrow 0^+} \alpha P_{\alpha}(E, K) \geq  \lim_{\alpha\rightarrow 0^+}  n  {V(K)^{\frac{n+\alpha}{n}}}{V(E)^{\frac{n+\alpha}{n}}}=n  V(E) V(K).$$ \end{proof} 
The definition for $\gamma_{\alpha}(K)$ and Theorem \ref{Minkowski type-1} imply that  
\begin{eqnarray*}
\frac{n}{\alpha} V(K)^{\frac{n+\alpha}{n}}  \leq \inf\{P_{\alpha}(E, K) V(E)^{-\frac{n-\alpha}{n}}: \ E\subset \Omega, V(E)>0\} = \gamma_{\alpha}(K). 
\end{eqnarray*} That is, we have a lower bound for $\gamma_{\alpha}(K)$:
$$
\gamma_{\alpha}(K)\geq \frac{n}{\alpha} V (K)^{\frac{n+\alpha}{n}}.
$$

\begin{rem}
\label{r22}  It is well known that the anisotropic isoperimetric inequality  (cf. \cite[(1.4)]{FMP}) \begin{equation}\label{anisotropic isoperimetric---1}
P(E,K)\ge nV(K)^\frac{1}{n}V(E)^\frac{n-1}{n}
\end{equation}  can be obtained by the classical Brunn-Minkowski inequality \cite{DI}.  However, such an inequality cannot be obtained from Theorem \ref{Minkowski type-1} by  letting $\alpha\to 1^-$, if one notices the second limit of (\ref{Monika-L2-1}). On the other hand,  inequalities in Theorem \ref{Minkowski type-1} and the anisotropic isoperimetric inequality have two common features: the dimension $n$ appears in front of the products of the powered volumes, and the sums of the powers of $V(K)$ and $V(E)$ are constants:
$$
\frac{n+\alpha}{n}+\frac{n-\alpha}{n}=2\quad\&\quad  \frac{1}{n}+\frac{n-1}{n}=1.
$$ 

 As in \cite{FMP}, it may be interesting to study the deficit:
$$
\frac{\alpha P_{\alpha}(E, K)}{{n} V (K)^{\frac{n+\alpha}{n}}  V(E)^{\frac{n-\alpha}{n}}}-1;
$$
see \cite{FMM} for a PDE-based treatment of such a question with $K=B^n_2$.  We leave this for future investigation. 
\end{rem}

 The relation between the anisotropic fractional Sobolev capacity and the anisotropic factional perimeter is stated in the following theorem, which is an extension of \cite[Theorem 2]{Xiao14a} for $K=B^n_2$.

\begin{thm}
\label{t2} Let $L$ be a compact subset of $\mathbb R^n$. Then 
$$
\mathrm{cap}(L; \dot{\Lambda}_{\alpha, K}^{1,1})=2\inf_{O\in\mathsf{O}^\infty(L)}P_\alpha(O, K),
$$ 
where $\mathsf{O}^\infty(L)$ denotes the class of all open sets with $C^\infty$  boundary that contain $L$.
\end{thm} 
\begin{proof} Let $L\subset\mathbb R^n$ be compact. For $f\in C_0^\infty$ with $f\ge \mathbf {1}_L$, one has 
$$
L\subset\{x\in\mathbb R^n: f(x)>t\}, \quad\forall t\in (0,1). 
$$
The generalized co-area formula in \cite{Vi} (see also \cite{L2})  implies \begin{eqnarray}\label{co-area-1}
\|f\|_{\dot{\Lambda}_{\alpha, K}^{1,1}}&=&2\int_{0}^\infty P_\alpha\big(\{x\in\mathbb R^n: f(x)>t\}, K\big)\,dt\\ \nonumber&\ge& 2\int_0^1 P_\alpha\big(\{x\in\mathbb R^n: f(x)>t\}, K \big)\,dt \\\nonumber &\ge& 2\inf_{O\in \mathsf{O}^\infty(L)}P_\alpha(O, K),
\end{eqnarray} where the last inequality follows from 
$$
\{x\in\mathbb R^n: f(x)>t\}\in \mathsf{O}^\infty(L).
$$ 
Hence, formula (\ref{Related capacity-0}) implies 
$$
\mathrm{cap}(L; \dot{\Lambda}_{\alpha, K}^{1,1})\ge 2\inf_{O\in \mathsf{O}^\infty(L)}P_\alpha(O, K).
$$ On the other hand, similar to the proof of Theorem 3.1 in \cite{HV} (or the proof of Part (ii) of Theorem \ref{equivalent perimeter geometric and analytic} in this paper), one can prove that $$
\mathrm{cap}(L; \dot{\Lambda}_{\alpha, K}^{1,1})\le \mathrm{cap}(\overline{O}; \dot{\Lambda}_{\alpha, K}^{1,1})\le 2P_\alpha(O, K), \quad\forall O\in \mathsf{O}^\infty(L),
$$ where the first inequality is by Part (ii) of Theorem \ref{t1}.  This further implies that 
$$
\mathrm{cap}(L; \dot{\Lambda}_{\alpha, K}^{1,1})\le 2\inf_{O\in \mathsf{O}^\infty(L)}P_\alpha(O, K), $$
and the desired formula for $\mathrm{cap}(L; \dot{\Lambda}_{\alpha, K}^{1,1})$ follows.
\end{proof}

\begin{rem}
\label{r23}  Combining formula  (\ref{function-inequality-1}) and the first limit of \cite[p.90, line 5]{L2},  we can prove the following co-area formula 
$$
\int_{\bbR^n} \|\nabla f(x) \|_{Z_1^*K}\,dx=2\int_{0}^\infty P\big(\{x\in\mathbb R^n: f(x)>t\}, Z_1K\big)\,dt.
$$ Moreover,  Theorem \ref{t2} together with formulas (\ref{function-inequality-1}),  (\ref{Sobolev:limit--1}) and (\ref{Monika-L2-1})  imply that 
\begin{equation}\label{Sobolev:inequality--1}
\mathrm{cap}(L; \dot{W}^{1,1}_{K})=2\inf_{O\in\mathsf{O}^\infty(L)}P(O, Z_1K),
\end{equation} 
which extends \cite[Lemma 2.2.5]{Maz} for $K=\ball$ to the anisotropic case.  
\end{rem}

We now establish the anisotropic isocapacitary inequality with fractional order $\alpha\in (0, 1)$. 
 \begin{cor}\label{ccc-1}  Let $L$ be a compact subset of $\bbR^n$.   Then, the following anisotropic isocapacitary inequality with fractional order $\alpha\in (0, 1)$ holds:
  $$  
  \alpha\mathrm{cap}(L; \dot{\Lambda}_{\alpha, K}^{1,1})\geq {2n}V (K)^{\frac{n+\alpha}{n}}  V(L)^{\frac{n-\alpha}{n}}.
  $$ 
    Moreover, this inequality is asymptotically optimal in the sense of:
 $$
 \lim_{\alpha\rightarrow 0^+}\alpha\mathrm{cap}(L; \dot{\Lambda}_{\alpha, K}^{1,1}) =\lim_{\alpha\rightarrow 0^+} 2{n} V (K)^{\frac{n+\alpha}{n}}  V(L)^{\frac{n-\alpha}{n}}=2n V(K)V(L). 
 $$ 
 \end{cor}  
\begin{proof} Combining Theorems  \ref{Minkowski type-1} and \ref{t2}, one has 
\begin{eqnarray*} 
\mathrm{cap}(L; \dot{\Lambda}_{\alpha, K}^{1,1})&=&2\inf_{O\in\mathsf{O}^\infty(L)}P_\alpha(O, K)\\
& \geq&   \inf_{O\in\mathsf{O}^\infty(L)}\bigg(2\gamma_{\alpha}(K)  V(O)^{\frac{n-\alpha}{n}}\bigg)\\ 
&\geq& 2 \gamma_{\alpha}(K)  V(L)^{\frac{n-\alpha}{n}}\\
& \geq&   \frac{2n}{\alpha}\cdot V (K)^{\frac{n+\alpha}{n}}  V(L)^{\frac{n-\alpha}{n}}.
\end{eqnarray*} 
Together with formula (\ref{Sobolev:limit--1}), one has
$$
2nV(L) V(K)=
\lim_{\alpha\to 0^+}\alpha\ \mathrm{cap}(L; \dot{\Lambda}_{\alpha, K}^{1,1})
\geq   \lim_{\alpha\rightarrow 0^+} {2n}  V (K)^{\frac{n+\alpha}{n}}  V(L)^{\frac{n-\alpha}{n}}=2nV(L) V(K).
$$ \end{proof}

\begin{rem}
\label{r24} Similarly, inequality (\ref{anisotropic isoperimetric---1}) and formula (\ref{Sobolev:inequality--1}) imply the following anisotropic isocapacitary inequality: 
$$
\mathrm{cap}(L; \dot{W}^{1,1}_{K})\ge 2 n V(Z_1K)^\frac{1}{n}V(L)^\frac{n-1}{n}.
$$ 
\end{rem}

 \section{Anisotropic fractional Sobolev embeddings}\label{s3}

This section dedicates to establish the anisotropic fractional Sobolev inequalities (generated by the Radon-measure-based-Lebesgue-space $L^{n/\beta}_\mu$ on $\mathbb R^n$) and their geometric counterparts for anisotropic fractional capacity. 

First, we have the anisotropic extension of \cite[Theorem 3(i)]{Xiao14a}.

\begin{thm}
\label{t3} Let $\mu$ be a nonnegative Radon measure on $\bbR^n$, and let $0<\beta<\infty$ and $\kappa_{n,\alpha,\beta}>0$ be constants. Then the following two inequalities are equivalent:  
\vskip 2mm 

\noindent {\rm(i)} The analytic inequality  
\begin{equation}\label{eq1}
\|f\|_{L_{\mu}^\frac{n}{\beta}}\le \kappa_{n,\alpha,\beta}\left(\int_0^\infty \bigg(\mathrm{cap}\big(\{x\in\mathbb R^n: |f(x)|\ge t\}; \dot{\Lambda}_{\alpha, K}^{1,1}\big)\bigg)^\frac{n}{\beta}\,dt^\frac{n}{\beta}\right)^\frac{\beta}{n}, \ \forall f\in C_0^\infty;
\end{equation} 

\noindent {\rm(ii)} The  geometric inequality  \begin{equation}\label{eq2}
\big(\mu(\overline{O})\big)^\frac{\beta}{n}\le \kappa_{n,\alpha,\beta}\, \mathrm{cap}(\overline{O};\dot{\Lambda}_{\alpha, K}^{1,1}), \ \   \mathrm{for\ all\ bounded\ domain}\ \ O\subset\mathbb R^n\ \mathrm{with}\ C^\infty\ \mathrm{boundary}\ \partial O.
\end{equation}
 
  \end{thm} 

\begin{proof}  By Fubini's theorem, one has, for all $f\in C_0^{\infty}$,  \begin{eqnarray}
\|f\|_{L_{\mu}^{\frac{n}{\beta}}}&=&\left(\int_{\bbR^n} |f(x)|^{\frac{n}{\beta}}\,d\mu(x)\right)^\frac{\beta}{n}  \nonumber
\\
&=&  \left(\int_{\bbR^n} \bigg[\int_0^{|f(x)|} {n}{\beta}^{-1} t^{\frac{n}{\beta}-1}\,dt\bigg]\,d\mu(x)\right)^\frac{\beta}{n} \nonumber\\  &=&\left(\int_{0}^{\infty}  \bigg[\int_{O_t(f)} {n}{\beta}^{-1} t^{\frac{n}{\beta}-1}\,d\mu(x)\bigg]\,dt\right)^\frac{\beta}{n}\nonumber\\ &=&\left(\int_0^\infty \mu\big(O_t(f)\big)\,dt^ \frac{n}{\beta}\right)^\frac{\beta}{n},  \label{level-equality---1} \end{eqnarray} 
where, for all $t>0$, $O_t(f)$ and $dt^ \frac{n}{\beta}$ are defined as
$$
O_t(f)=\{x\in\bbR^n:\ |f(x)|> t\}\quad\&\quad dt^ \frac{n}{\beta}={n}{\beta}^{-1} t^{\frac{n}{\beta}-1}\,dt.
$$ 

\vskip 2mm \noindent (ii) $\Rightarrow$ (i)  Suppose that inequality (\ref{eq2}) holds.  Note that, for  $f\in C_0^\infty$, the set  $
O_t(f) $   is a bounded open domain with $C^\infty$ boundary. Together with  inequality (\ref{eq2}) and formula (\ref{level-equality---1}), one gets the desired inequality (\ref{eq1}) as follows: \begin{eqnarray*} 
\|f\|_{L_{\mu}^ \frac{n}{\beta}}&=&\left(\int_0^\infty \mu\big(O_t(f)\big)\,dt^ \frac{n}{\beta}\right)^\frac{\beta}{n}\\
&\leq&  \left(\int_0^\infty \mu\big(\overline{O_t(f)}\big)\,dt^ \frac{n}{\beta}\right)^\frac{\beta}{n} \\ 
&\le&\kappa_{n,\alpha,\beta}\left(\int_0^\infty\Big(\mathrm{cap}\big(\overline{O_t(f)};\dot{\Lambda}_{\alpha, K}^{1,1} \big)\Big)^\frac{n}{\beta}\,dt^\frac{n}{\beta}\right)^\frac{\beta}{n}.
\end{eqnarray*}
 
\noindent (i) $\Rightarrow$ (ii) Suppose that inequality (\ref{eq1}) holds. For any bounded domain $O\subset\mathbb R^n$ with $C^\infty$ boundary $\partial O$ and $0<\epsilon<1$,   let $$
  f_\epsilon(x)=\left\{  \begin{array}{ll}
 1-\epsilon^{-1}\mathrm{dist}(x,\overline{O}), & \mathrm{if\  dist}(x,\overline{O})<\epsilon  \\ 
  0, & \mathrm{if\ dist}(x, \overline{O})\ge\epsilon  \\ 
  \end{array}\right.
  $$ where $\mathrm{dist}(x,E)$  denotes the Euclidean distance of a point $x$ to a set $E$. One can check that $f_\epsilon\in C^{\infty}_0$ and hence  inequality  (\ref{eq1}) holds for $f_\epsilon$. Moreover, 
  \begin{equation} \label{mu-equality-1} 
  \big(\mu(\overline{O})\big)^\frac{\beta}{n} =  \lim_{\epsilon\rightarrow 0^+} \|f_\epsilon\|_{L_{\mu}^{\frac{n}{\beta}}}. 
  \end{equation}  
  Let 
$
O_\epsilon=\{x\in\mathbb R^n:\ \mathrm{dist}(x,\overline{O})<\epsilon\}.
$
Inequality (\ref{eq1}) implies that  for all $0<\epsilon<1$,
\begin{eqnarray*}
 \|f_\epsilon\|_{L_{\mu}^{\frac{n}{\beta}}}&\leq& \kappa_{n,\alpha,\beta}
\left(\int_0^\infty \bigg(\mathrm{cap}\big(\overline{O_t(f_\epsilon)}; \dot{\Lambda}_{\alpha, K}^{1,1}\big)\bigg)^\frac{n}{\beta}\,dt^\frac{n}{\beta}\right)^\frac{\beta}{n}\\
&=&\kappa_{n,\alpha,\beta}
\left(\int_0^1 \big(\mathrm{cap}(\overline{O_t(f_\epsilon)}; \dot{\Lambda}_{\alpha, K}^{1,1})\big)^\frac{n}{\beta}\,dt^\frac{n}{\beta}\right)^\frac{\beta}{n}\\
 &\le&\kappa_{n,\alpha,\beta}\mathrm{cap}(\overline{O_\epsilon};\dot{\Lambda}_{\alpha, K}^{1,1}),
\end{eqnarray*} where the last inequality is due to Part (ii) of Theorem \ref{t1} and $\overline{O_t(f_\epsilon)}\subset \overline{O_\epsilon}$. Taking $\epsilon\rightarrow 0^+$, one gets inequality (\ref{eq2}) by  Part  (iv) of Theorem \ref{t1} and formula (\ref{mu-equality-1}). \end{proof}

  As a matter of fact, both inequalities (\ref{eq1}) and (\ref{eq2}) hold true for $\mu$ being the Lebesgue measure on $\bbR^n$ with constant $\kappa_{n,\alpha,n-\alpha}=\big({2\gamma_{\alpha} (K)}\big)^{-1}.$
Moreover, if the nonnegative Radon measure $\mu$ is absolutely continuous with respect to the Lebesgue measure on $\bbR^n$ and  $f(x)=\frac{\,d\mu}{\,dx}$ is  bounded on $\bbR^n$, then inequalities (\ref{eq1}) and (\ref{eq2}) hold true for some constant $\kappa_{n,\alpha,n-\alpha}$.   To this end,  it can be seen from the proof of Corollary \ref{ccc-1}  that for all  bounded domain $O\subset\mathbb R^n$ with $C^\infty$ boundary $ \partial O$,  
  \begin{equation*} 
\big(V({O})\big)^\frac{n-\alpha}{n}=\big(V(\overline{O})\big)^\frac{n-\alpha}{n}\le \big({2\gamma_{\alpha} (K)}\big)^{-1} \mathrm{cap}(\overline{O};\dot{\Lambda}_{\alpha, K}^{1,1}). \end{equation*} That is, inequality (\ref{eq2}) holds true with constant 
 $
\kappa_{n,\alpha,n-\alpha}=(2\gamma_{\alpha} (K))^{-1},
 $   
and so does inequality (\ref{eq1}) by Theorem \ref{t3}. Moreover, let $\mu$ be such that  $f(x)=\frac{\,d\mu}{\,dx}$  is  bounded on $\bbR^n$, say by $M<\infty$. For all bounded domain $O\subset\mathbb R^n$ with $C^\infty$ boundary $ \partial O$, one has,  $\mu(\overline{O})\leq M V(\overline{O})$,  and hence, \begin{equation*} 
\big(\mu(\overline{O})\big)^\frac{n-\alpha}{n}\leq M^\frac{n-\alpha}{n}  \big(V(\overline{O})\big)^\frac{n-\alpha}{n}\le {M^\frac{n-\alpha}{n} }\big({2\gamma_{\alpha} (K)}\big)^{-1} \, \mathrm{cap}(\overline{O};\dot{\Lambda}_{\alpha, K}^{1,1}). 
\end{equation*} 
That is, inequality (\ref{eq2}) holds true for $\mu$ with constant 
$ 
\kappa_{n,\alpha,n-\alpha}= M^\frac{n-\alpha}{n} (2\gamma_{\alpha} (K))^{-1},
$    
and so does inequality (\ref{eq1}) by Theorem \ref{t3}. 

\begin{rem}
\label{r31} Similar to Theorem \ref{t3} and comments after, one can get analogous results for the  anisotropic fractional Sobolev capacity $\mathrm{cap}(\cdot, \dot{W}_K^{1,1})$. More precisely, with $\mu$ and $\beta$ as in Theorem \ref{t3} and  $\kappa_{n, \beta}$ a constant,  the following two inequalities are equivalent:  
 \vskip 2mm \noindent
 (i) For all $f\in C_0^\infty$,
 $$
\|f\|_{L_{\mu}^\frac{n}{\beta}} \le \kappa_{n, \beta}\left(\int_0^\infty \bigg(\mathrm{cap}\big(\{x\in\mathbb R^n: |f(x)|\ge t\};
 \dot{W}_{K}^{1,1}\big)\bigg)^\frac{n}{\beta}\,dt^\frac{n}{\beta}\right)^\frac{\beta}{n}; 
$$ 
\noindent (ii) For  all bounded  domain $O\subset\mathbb R^n$ with $C^\infty$ boundary  $\partial O$,
$$
 \big(\mu(\overline{O})\big)^\frac{\beta}{n} \le\kappa_{n, \beta}\, \mathrm{cap}(\overline{O};\dot{W}_{K}^{1,1}).
$$
 Moreover, the above inequalities hold for $\mu$ being the Lebesgue measure on $\bbR^n$ with constant $\kappa_{n, n-1}=\big(2 n V(Z_1K)^\frac{1}{n}\big)^{-1}.$ \end{rem}

Second, we have the anisotropic version of \cite[Theorem 3 (ii)]{Xiao14a}.

\begin{thm} \label{equivalent perimeter geometric and analytic} Let $0<\beta<\infty$. The following inequalities hold and are equivalent: 

 \vskip 2mm 
 
 \noindent {\rm(i)} 
 The  analytic inequality 
 \begin{equation}\label{eq3}
\left(\int_0^\infty \bigg(\mathrm{cap}\big(\{x\in\mathbb R^n: |f(x)|\ge t\}; \dot{\Lambda}_{\alpha, K}^{1,1}\big)\bigg)^\frac{n}{\beta}\,dt^\frac{n}{\beta}\right)^\frac{\beta}{n}\le\|f\|_{\dot{\Lambda}^{1,1}_{\alpha, K}}, \ \ \ \ \forall f\in C_0^\infty;
\end{equation}

\noindent {\rm(ii)} The  geometric inequality 
\begin{equation}\label{eq4}
\mathrm{cap}(\overline{O};\dot{\Lambda}_{\alpha, K}^{1,1})\le 2 P_\alpha(\overline{O}, K), \quad  \mathrm{for\ all\ bounded\ domain}\ \ O\subset\mathbb R^n\ \mathrm{with}\ C^\infty\ \mathrm{boundary}\ \ \partial O.
\end{equation}
 \end{thm}

\begin{proof}   We first prove that  inequality (\ref{eq4}) holds and is equivalent to inequality (\ref{eq3}), and hence inequality  (\ref{eq3}) holds automatically.

\vskip 2mm   The proof of  inequality (\ref{eq4}) is similar to that of Theorem 3.1 in \cite{HV}. For completeness, we include a brief proof here. Let $O\subset\mathbb R^n$ be a bounded domain with $C^\infty$ boundary $\partial O$. Recall that $\|\cdot\|$ is equivalent to $\|\cdot\|_K$ for any given origin-symmetric convex body $K$.  By Lemma 3.2 in \cite{HV},  for all $\epsilon>0$, one can find a function $g\in C^{\infty}_0$, such that, $0\leq g\leq 1$,  $g(x)=1$ for $x\in \overline{O}$ (which implies $g\geq \mathbf{1}_{\overline{O}}$), and $$\int_{ O^c} \int_{O^c} \frac{|g(x)-g(y)|}{\|x-y\|_K^{n+\alpha}}\,dx\,dy<\epsilon.$$
Hence, formulas (\ref{Related capacity-0}) and  (\ref{smooth-perimater}),   together with $g\in C^{\infty}_0$ and $g\geq \mathbf{1}_{\overline{O}}$, imply    \begin{eqnarray*} 
\mathrm{cap}(\overline{O}; \dot{\Lambda}_{\alpha, K}^{1,1})&\leq& \int_{ \bbR^n} \int_{\bbR^n} \frac{|g(x)-g(y)|}{\|x-y\|_K^{n+\alpha}}\,dx\,dy\\
&\leq & 2\int_{ O} \int_{O^c} \frac{|g(x)-g(y)|}{\|x-y\|_K^{n+\alpha}}\,dx\,dy+\int_{O^c} \int_{O^c} \frac{|g(x)-g(y)|}{\|x-y\|_K^{n+\alpha}}\,dx\,dy\\&<& 2 P_{\alpha}(O, K)+\epsilon\\
&=&2 P_{\alpha}(\overline{O}, K)+\epsilon.
\end{eqnarray*} The desired inequality (\ref{eq4}) follows by taking $\epsilon\rightarrow 0^+$.

 \vskip 2mm  Now we prove the equivalence between inequalities (\ref{eq3}) and (\ref{eq4}).  First, we assume that  inequality (\ref{eq3}) holds true.  Let $\epsilon\in (0,1)$ and  $O\subset\mathbb R^n$ be  a bounded domain  with $C^\infty$ boundary $\partial O$.   Let $O_\epsilon$ and $f_\epsilon$ be as in the proof of Theorem \ref{t3}.   Also note that  $f_\epsilon(x)=1$ for all $x\in \overline{O}$, and hence $\overline{O}\subset \overline{O_t(f_\epsilon)}$ for all $\epsilon\in (0, 1)$ and $t\in (0, 1)$. By Part (ii) of Theorem \ref{t1} and inequality (\ref{eq3}), one has  
 \begin{eqnarray*} \mathrm{cap}(\overline{O}; \dot{\Lambda}_{\alpha, K}^{1,1})&\le&\left(\int_0^1 \bigg(\mathrm{cap}\big(\overline{O_t(f_\epsilon)}; \dot{\Lambda}_{\alpha, K}^{1,1}\big)\bigg)^\frac{n}{\beta}\,dt^\frac{n}{\beta}\right)^\frac{\beta}{n}\\
 &\le&\left(\int_0^{\infty} \bigg(\mathrm{cap}\big(\overline{O_t(f_\epsilon)}; \dot{\Lambda}_{\alpha, K}^{1,1}\big)\bigg)^\frac{n}{\beta}\,dt^\frac{n}{\beta}\right)^\frac{\beta}{n}\\
 &\leq&  \|f_\epsilon\|_{\dot{\Lambda}^{1,1}_{\alpha, K}}. 
 \end{eqnarray*}  
 As $f_{\epsilon}(x)\rightarrow \mathbf{1}_{\overline{O}}$, the dominated convergent theorem implies the desired inequality (\ref{eq4}): \begin{eqnarray*} \mathrm{cap}(\overline{O}; \dot{\Lambda}_{\alpha, K}^{1,1}) \leq \lim_{\epsilon\rightarrow 0^+} \|f_\epsilon\|_{\dot{\Lambda}^{1,1}_{\alpha, K}}=
 \|\mathbf{1}_{\overline{O}}\|_{\dot{\Lambda}^{1,1}_{\alpha, K}}=2P_\alpha(\overline{O}, K).\end{eqnarray*} 

Second, we assume that  inequality (\ref{eq4}) holds.  Note that $O_t(f)\subset O_s(f)$  holds for any function $f\in C^\infty_0$ and $0<s<t$. Part (ii) of Theorem \ref{t1} implies that $\mathrm{cap}\big(\overline{O_t(f)};\ \dot{\Lambda}^{1,1}_{\alpha, K}\big)$ is decreasing  on $t\in [0,\infty)$.  Hence,  

\begin{eqnarray*} 
 t^{\frac{n}{\beta}-1}\bigg(\mathrm{cap}\big(\overline{O_t(f)}; \dot{\Lambda}^{1,1}_{\alpha, K}\big)\bigg)^\frac{n}{\beta} 
&=&\bigg(t\ \mathrm{cap}\big(\overline{O_t(f)};\ \dot{\Lambda}^{1,1}_{\alpha, K}\big)\bigg)^{\frac{n}{\beta}-1}\mathrm{cap}\big(\overline{O_t(f)};\ \dot{\Lambda}^{1,1}_{\alpha, K} \big)\\
&\le& \left(\int_0^t\mathrm{cap}\big(\overline{O_s(f)};\ \dot{\Lambda}^{1,1}_{\alpha, K}\big)\,ds\right)^{\frac{n}{\beta}-1}\mathrm{cap}\big(\overline{O_t(f)};\ \dot{\Lambda}^{1,1}_{\alpha, K}\big)\\
&=& \frac{\beta}{n} \cdot  \frac{d}{dt}\left(\int_0^t \mathrm{cap}\big(\overline{O_s(f)};\ \dot{\Lambda}^{1,1}_{\alpha, K}\big)\,ds\right)^{\frac{n}{\beta}}. \end{eqnarray*} 
Integrating the above inequality over $t\in (0, \infty)$, one has  
\begin{eqnarray*} 
\int_0^\infty \bigg(\mathrm{cap}\big(\overline{O_t(f)};\ \dot{\Lambda}^{1,1}_{\alpha, K}\big) \bigg)^{\frac{n}{\beta}}\,dt^\frac n{\beta} &=&  \frac{n}{\beta} \cdot \int_0^{\infty} 
   t^{\frac{n}{\beta}-1}\bigg(\mathrm{cap}\big(\overline{O_t(f)}; \dot{\Lambda}^{1,1}_{\alpha, K}\big)\bigg)^\frac{n}{\beta}\,dt \\ 
&\leq& \int_0^{\infty} \frac{d}{dt}\left(\int_0^t \mathrm{cap}\big(\overline{O_s(f)};\ \dot{\Lambda}^{1,1}_{\alpha, K}\big)\,ds\right)^{\frac{n}{\beta}}\,dt\\
&=&\left(\int_0^\infty \mathrm{cap}\big(\overline{O_s(f)};\ \dot{\Lambda}^{1,1}_{\alpha, K}\big)\,ds\right)^{\frac{n}{\beta}}.
\end{eqnarray*}  
Hence, inequality  (\ref{eq4}) and the co-area formula (\ref{co-area-1}) imply the desired inequality (\ref{eq3}):  
\begin{eqnarray*} 
\left(\int_0^\infty \bigg(\mathrm{cap}\big(\overline{O_t(f)};\ \dot{\Lambda}^{1,1}_{\alpha, K}\big)\bigg)^{\frac{n}{\beta}}\,dt^\frac n{\beta}\right)^\frac{\beta}{n}&\le&
 \int_0^\infty\mathrm{cap}\big(\overline{O_t(f)};\dot{\Lambda}^{1,1}_{\alpha, K}\big)\,dt\\
 &\leq&  2\int_0^\infty P_\alpha\big(\overline{O_t(f)}, K\big)\,dt\\
 &=&2\int_0^\infty P_\alpha\big( {O_t(f)}, K\big)\,dt\\
 &=& \|f\|_{\dot{\Lambda}_{\alpha, K}^{1,1}}.  
\end{eqnarray*} 
\end{proof}

\begin{rem}
\label{r32} Similar result for anisotropic Sobolev capacity  $\mathrm{cap}(\cdot, \dot{W}_K^{1,1})$ also holds and is an extension of \cite[Theorem 1.1]{Xiao07}.   More precisely, with $0<\beta<n$,  the following inequalities hold and are equivalent:  

\vskip 2mm \noindent 
(i) For all $f\in C_0^\infty$, $$
\left(\int_0^\infty \bigg(\mathrm{cap}\big(\{x\in\mathbb R^n: |f(x)|\ge t\}; \dot{W}_{K}^{1,1}\big)\bigg)^\frac{n}{\beta}\,dt^\frac{n}{\beta}\right)^\frac{\beta}{n}\le \int_{\bbR^n} \|\nabla f(x)\|_{Z_1^*K}\,dx; 
$$

\noindent (ii) For  all  bounded  domain $O\subset\mathbb R^n$  with $C^\infty$  boundary  $\partial O$,
$$
\mathrm{cap}(\overline{O};\dot{W}_{K}^{1,1})\le 2 P (\overline{O}, Z_1K).
$$ \end{rem}

Finally, as a more general formulation of \cite[Theorem 4]{Xiao14a} and \cite[Theorem 9]{L2}, we have the following equivalence.

\begin{thm}\label{t4} Let $\mu$ be a nonnegative Radon measure on $\bbR^n$, and  $0<\beta\le n$ and $\kappa_{n,\alpha,\beta}>0$ are constants. The following three inequalities are equivalent:
 
\vskip 2mm 

\noindent {\rm(i)} The anisotropic fractional Sobolev inequality  $$\|f\|_{L_{\mu}^\frac{n}{\beta}}\le \kappa_{n,\alpha,\beta}\|f\|_{\dot{\Lambda}_{\alpha, K}^{1,1}}, \ \  \mathrm{for\ all}\  f\in C_0^\infty;
$$ 

\noindent {\rm(ii)}  The anisotropic fractional isocapacitary inequality $$\big(\mu(\overline{O})\big)^\frac{\beta}{n}\le \kappa_{n,\alpha,\beta}\mathrm{cap}(\overline{O},\dot{\Lambda}_{\alpha, K}^{1,1}), \ \ \mathrm{for\ any\ bounded\ domain}\  O\subset\mathbb R^n \ \mathrm{with}\  C^\infty\ \mathrm{boundary}\ \partial O;$$ 

\noindent {\rm(iii)} The anisotropic  fractional isoperimetric inequality $$\big(\mu(\overline{O})\big)^\frac{\beta}{n}\le 2\kappa_{n,\alpha,\beta}P_\alpha(\overline{O}, K), \ \ \mathrm{for\ any\ bounded\ domain}\  O\subset\mathbb R^n\  \mathrm{with}\ C^\infty\ \mathrm{boundary}\  \partial O.$$ 
\end{thm}

\begin{proof} (i)$\Rightarrow$(ii) Suppose that the anisotropic fractional Sobolev inequality in (i) holds true.  Then, for all $f\in C_0^\infty$ with $f\geq \mathbf{1}_{\overline{O}}, $ one has 
\begin{eqnarray*}
\big(\mu(\overline{O})\big)^\frac{\beta}{n}= \left(\int_{\bbR^n} \mathbf{1}_{\overline{O}} \,d\mu(x) \right)^\frac{\beta}{n}  \leq    \left(\int_{\bbR^n} f(x)^\frac{n}{\beta}\,d\mu(x) \right)^\frac{\beta}{n}  = \|f\|_{L_{\mu}^\frac{n}{\beta}} \leq \kappa_{n,\alpha,\beta}\|f\|_{\dot{\Lambda}_{\alpha, K}^{1,1}}. 
\end{eqnarray*} 
Taking the infimum over $f\in C_0^\infty$ with $f\geq \mathbf{1}_{\overline{O}}$  and by formula  (\ref{Related capacity-0}),  one gets the desired anisotropic fractional isocapacitary inequality  
\begin{eqnarray*}
\big(\mu(\overline{O})\big)^\frac{\beta}{n}\leq  \kappa_{n,\alpha,\beta}\mathrm{cap}(\overline{O},\dot{\Lambda}_{\alpha, K}^{1,1}). 
\end{eqnarray*}

\noindent (ii)$\Rightarrow$(iii) Assume that the anisotropic fractional isocapacitary inequality holds. Then, for any bounded domain $O\subset\mathbb R^n$ with $C^\infty$ boundary $\partial O$,   one gets the desired  anisotropic  fractional isoperimetric inequality: 
\begin{eqnarray*}   
\big(\mu(\overline{O})\big)^\frac{\beta}{n}\le \kappa_{n,\alpha,\beta}\mathrm{cap}(\overline{O},\dot{\Lambda}_{\alpha, K}^{1,1})  \le   2  \kappa_{n,\alpha,\beta} P_\alpha(\overline{O}, K) 
\end{eqnarray*} 
where the last inequality follows from inequality  (\ref{eq4}).  
 
\smallskip
 
\noindent  (iii)$\Rightarrow$(i)  Assume that the anisotropic  fractional isoperimetric inequality  
holds. Let $f\in C_0^\infty$ and $O_t(f)=\{x\in\bbR^n:\ |f(x)|> t\}$ for all $t\geq 0.$
Obviously, $\mu(O_t(f))$ is a decreasing function on $t\in [0,\infty)$, and hence for $0<\beta\le n$,  
\begin{eqnarray*}   
 \left(\int_0^t \mu \big(O_s(f)\big)\,ds^\frac{n}{\beta}\right)^{\frac{\beta}{n}-1}\mu \big(O_t(f)\big)t^\frac{n}{\beta} \leq  \left(\int_0^t \mu \big(O_t(f)\big)\,ds^\frac{n}{\beta}\right)^{\frac{\beta}{n}-1}\mu \big(O_t(f)\big)t^{\frac{n}{\beta}} 
 = \Big(\mu \big(O_t(f)\big)\Big)^\frac{\beta}{n} t.   
 \end{eqnarray*}  
 Together with equality (\ref{level-equality---1}), one has 
  \begin{eqnarray*} \|f\|_{L_{\mu}^\frac{n}{\beta}} 
&=&\left(\int_0^\infty \mu \big(O_t(f)\big)\,dt^\frac{n}{\beta}\right)^\frac{\beta}{n}\\
&=&\int_0^\infty\frac{d}{dt}\left(\int_0^t \mu \big(O_s(f)\big)\,ds^\frac{n}{\beta}\right)^\frac{\beta}{n}\,dt\\
&=&\int_0^\infty\left(\int_0^t \mu \big(O_s(f)\big)\,ds^\frac{n}{\beta}\right)^{\frac{\beta}{n}-1}\mu \big(O_t(f)\big)t^{\frac{n}{\beta}-1}\,dt\\
&\le&\int_0^\infty \Big(\mu \big(O_t(f)\big)\Big)^\frac{\beta}{n}\,dt. 
\end{eqnarray*}  
Employing the anisotropic  fractional isoperimetric inequality to $O_t(f)$, together with formulas (\ref{smooth-perimater}) and (\ref{co-area-1}), one gets, for all $f\in C^{\infty}_0$,   
\begin{eqnarray*} 
\|f\|_{L_{\mu}^\frac{n}{\beta}} 
 \le \int_0^\infty \Big(\mu \big(\overline{O_t(f)}\big)\Big)^\frac{\beta}{n}\,dt \le  2\kappa_{n,\alpha,\beta}\int_0^\infty P_\alpha\big(O_t(f), K\big)\,dt  
 =  \kappa_{n,\alpha,\beta}\|f\|_{\dot{\Lambda}^{1,1}_{\alpha, K}}, 
 \end{eqnarray*} 
 the desired anisotropic fractional Sobolev inequality.  
 \end{proof}
\begin{rem}
\label{r33} Similarly, for a nonnegative Radon measure $\mu$, constants $0<\beta\le n$ and $\kappa_{n,\beta}>0$, the following three inequalities are equivalent, whence extending \cite[Proposition 3.1]{Xiao09} (cf. \cite[Propisition 3.1]{AX}):
\vskip 2mm \noindent (i) For  all $f\in C_0^\infty$,
$$ 
\|f\|_{L_{\mu}^\frac{n}{\beta}}\le \kappa_{n, \beta} \int_{\bbR^n} \|\nabla f(x)\|_{Z_1^*K} \,dx;
$$
 \noindent (ii) For  any  bounded  domain $O\subset\mathbb R^n$ with $C^\infty$ boundary $\partial O$;
$$ 
  \big(\mu(\overline{O})\big)^\frac{\beta}{n} \le \kappa_{n, \beta} \mathrm{cap}(\overline{O},\dot{W}_{K}^{1,1});
$$
\noindent (iii)   For any  bounded  domain  $O\subset\mathbb R^n$ with $C^\infty$  boundary  $\partial O$,
$$
  \big(\mu(\overline{O})\big)^\frac{\beta}{n} \le 2\kappa_{n, \beta} P(\overline{O}, Z_1K).
$$
   \end{rem}

\end{document}